\documentclass[a4paper,11pt]{amsart}
\pdfoutput=1
\usepackage{amssymb,amsmath,amsthm,amsfonts,centernot}
\usepackage[colorlinks=true,linkcolor=blue,citecolor=blue,urlcolor=cyan, pdfpagelabels=false]{hyperref}
\usepackage{amsmath,tikz}
\usepackage{dsfont}
\usepackage{diagbox}
\begin{document}
\numberwithin{equation}{section}
\newtheorem{theorem}{Theorem}
\newtheorem*{theo}{Theorem}
\newtheorem{algo}{Algorithm}
\newtheorem{lem}{Lemma} 
\newtheorem{de}{Definition} 
\newtheorem{ex}{Example}
\newtheorem{pr}{Proposition} 
\newtheorem{claim}{Claim} 
\newtheorem*{re}{Remark}
\newtheorem*{asi}{Aside}
\newtheorem{co}{Corollary}
\newtheorem{conv}{Convention}
\newcommand{\di}{\hspace{1.5pt} \big|\hspace{1.5pt}}
\newcommand{\idi}{\hspace{.5pt}|\hspace{.5pt}}
\newcommand{\hs}{\hspace{1.3pt}}
\newcommand{\thmf}{Theorem~1.15$'$}
\newcommand{\ndi}{\centernot{\big|}}
\newcommand{\nidi}{\hspace{.5pt}\centernot{|}\hspace{.5pt}}
\newcommand{\lpp}{\mbox{$\hspace{0.12em}\shortmid\hspace{-0.62em}\alpha$}}
\newcommand{\btt}{\mbox{$\raisebox{-0.59ex}
  {${{l}}$}\hspace{-0.215em}\beta\hspace{-0.88em}\raisebox{-0.98ex}{\scalebox{2}
  {$\color{white}.$}}\hspace{-0.416em}\raisebox{+0.88ex}
  {$\color{white}.$}\hspace{0.46em}$}{}}
  \newcommand{\un}{\hs\underline{\hspace{5pt}}\hs}
\newcommand{\lp}{\widehat{\lpp}}
\newcommand{\bt}{\hspace*{2pt}\widehat{\hspace*{-2pt}\btt}}
\newcommand{\PQ}{\bb{P}^1(\bb{Q})}
\newcommand{\pmn}{\cl{P}_{M,N}}
\newcommand{\he}{holomorphic eta quotient\hspace*{2.5pt}}
\newcommand{\hes}{holomorphic eta quotients\hspace*{2.5pt}}
\newcommand{\defG}{Let $G\subset\GG$ be a subgroup that is conjugate to a finite index subgroup of $\G$. } 
\newcommand{\defg}{Let $G\subset\GG$ be a subgroup that is conjugate to a finite index subgroup of $\G$\hs\hs} 
\renewcommand{\phi}{\varphi}
\newcommand{\Z}{\bb{Z}}
\newcommand{\ZD}{\Z^{\D}}
\newcommand{\N}{\bb{N}}
\newcommand{\Q}{\bb{Q}}
\newcommand{\A}{\widehat{A}}
\newcommand{\pii}{{{\pi}}}
\newcommand{\R}{\bb{R}}
\newcommand{\C}{\bb{C}}
\newcommand{\I}{\hs\cl{I}_{n,N}}
\newcommand{\St}{\operatorname{Stab}}
\newcommand{\D}{\cl{D}_N}
\newcommand{\rh}{{{\boldsymbol\rho}}}
\newcommand{\bh}{{\cl{M}}} 
\newcommand{\lv}{\hyperlink{level}{{\text{level}}}\hspace*{2.5pt}}
\newcommand{\fct}{\hyperlink{factor}{{\text{factor}}}\hspace*{2.5pt}}
\newcommand{\q}{\hyperlink{q}{{\mathbin{q}}}}
\newcommand{\rd}{\hyperlink{redu}{{{\text{reducible}}}}\hspace*{2.5pt}}
\newcommand{\ird}{\hyperlink{irredu}{{{\text{irreducible}}}}\hspace*{2.5pt}}
\newcommand{\str}{\hyperlink{strong}{{{\text{strongly reducible}}}}\hspace*{2.5pt}}
\newcommand{\rdn}{\hyperlink{redon}{{{\text{reducible on}}}}\hspace*{2.5pt}}
\newcommand{\atl}{\hyperlink{atinv}{{\text{Atkin-Lehner involution}}}\hspace*{3.5pt}}
\newcommand{\T}{\mathrm{T}}
\newcommand{\nm}{{N,M}}
\newcommand{\mn}{{M,N}}
\renewcommand{\H}{\fr{H}}
\newcommand{\W}{\text{\calligra W}_n}
\newcommand{\GG}{\cl{G}}
\newcommand{\g}{\fr{g}}
\newcommand{\Gm}{\Gamma}
\newcommand{\Gmtl}{\widetilde{\Gamma}_\ell}
\newcommand{\gm}{\gamma}
\newcommand{\go}{\gamma_1}
\newcommand{\gmt}{\widetilde{\gamma}}
\newcommand{\gmdt}{\widetilde{\gamma}'}
\newcommand{\gmot}{\widetilde{\gamma}_1}
\newcommand{\gmdot}{{\widetilde{\gamma}}'_1}
\newcommand{\s}{\Large\text{{\calligra r}}\hspace{1.5pt}}
\newcommand{\ms}{m_{{{S}}}}
\newcommand{\nisim}{\centernot{\sim}}
\newcommand{\level}{\hyperlink{level}{{\text{level}}}}
\newcommand{\Redcon}{the \hyperlink{red}{\text{Reducibility~Conjecture}}}
\newcommand{\Conred}{Conjecture~$1'$}
\newcommand{\Conredd}{Conjecture~$1''$}
\newcommand{\Conreddd}{Conjecture~$1'''$}
\newcommand{\Conired}{Conjecture~$2'$}
\newtheorem*{pro}{\textnormal{\textit{Proof of the proposition}}}
\newtheorem*{cau}{Caution}
\newtheorem*{conjec}{Conjecture}
\newtheorem{thrmm}{Theorem}[section]
\newtheorem{no}{Notation}
\renewcommand{\thefootnote}{\fnsymbol{footnote}}
\newtheorem{oq}{Open question}
\newtheorem{conj}{Conjecture}
\newtheorem{hy}{Hypothesis} 
\newtheorem{expl}{Example}
\newcommand\ileg[2]{\bigl(\frac{#1}{#2}\bigr)}
\newcommand\leg[2]{\Bigl(\frac{#1}{#2}\Bigr)}
\newcommand{\e}{\eta}
\newcommand{\sgn}{\operatorname{sgn}}
\newcommand{\bb}{\mathbb}
\newtheorem*{conred}{Conjecture~\ref{con1}$\mathbf{'}$}
\newtheorem*{conredd}{Conjecture~\ref{con1}$\mathbf{''}$}
\newtheorem*{conreddd}{Conjecture~\ref{con1}$\mathbf{'''}$}
\newtheorem*{conired}{Conjecture~\ref{19.1Aug}$\mathbf{'}$}
\newtheorem*{procl}{\textnormal{\textit{Proof}}}
\newtheorem*{thmbb}{Theorem~\ref{b1}$\mathbf{'}$}
\newcommand{\thmb}{Theorem~1$'$}
\newtheorem*{coo}{Corollary~\ref{17Aug}$\mathbf{'}$}
\newcommand{\cooo}{Corollary~\ref{17Aug}$'$}
\newtheorem*{cotw}{Corollary~\ref{17.1Aug}$\mathbf{'}$}
\newcommand{\cotww}{Corollary~\ref{17.1Aug}$'$}
\newtheorem*{cotb}{Corollary~\ref{b11}$\mathbf{'}$}
\newcommand{\cotbb}{Corollary~\ref{17.2Aug}$'$}
\newtheorem*{cne}{Corollary~\ref{15.5Aug}$\mathbf{'}$}
\newcommand{\cnew}{Corollary~\ref{15.5Aug}$'$\hspace{3.5pt}}
\newcommand{\fr}{\mathfrak}
\newcommand{\cl}{\mathcal}
\newcommand{\rad}{\mathrm{rad}}
\newcommand{\ord}{\operatorname{ord}}
\newcommand{\m}{\setminus}
\newcommand{\G}{\Gamma_1}
\newcommand{\GN}{\Gamma_0(N)}
\newcommand{\X}{\widetilde{X}}
\renewcommand{\P}{{\textup{p}}} 
\newcommand{\al}{{\hs\operatorname{al}}}
\newcommand{\p}{p_\text{\tiny (\textit{N})}}
\newcommand{\pN}{p_\text{\tiny\textit{N}}}
\newcommand{\U}{u_\textit{\tiny N}}
\newcommand{\Upr}{u_{\textit{\tiny N}^\prime}}
\newcommand{\Up}{u_{\textit{\tiny p}^\textit{\tiny e}}}
\newcommand{\Un}{u_{\textit{\tiny p}_\textit{\tiny 1}^{\textit{\tiny e}_\textit{\tiny 1}}}}
\newcommand{\Um}{u_{\textit{\tiny p}_\textit{\tiny m}^{\textit{\tiny e}_\textit{\tiny m}}}}
\newcommand{\Ut}{u_{\text{\tiny 2}^\textit{\tiny a}}}
\newcommand{\At}{A_{\text{\tiny 2}^\textit{\tiny a}}}
\newcommand{\Uh}{u_{\text{\tiny 3}^\textit{\tiny b}}}
\newcommand{\Ah}{A_{\text{\tiny 3}^\textit{\tiny b}}}
\newcommand{\Uprl}{u_{\textit{\tiny N}_1}}
\newcommand{\Uprlm}{u_{\textit{\tiny N}_i}}
\newcommand{\UM}{u_\textit{\tiny M}}
\newcommand{\UMp}{u_{\textit{\tiny M}_1}}
\newcommand{\w}{\omega_\textit{\tiny N}}
\newcommand{\wm}{\omega_\textit{\tiny M}}
\newcommand{\wa}{\omega_{\text{\tiny N}_\textit{\tiny a}}}
\newcommand{\wma}{\omega_{\text{\tiny M}_\textit{\tiny a}}}
\renewcommand{\P}{{\textup{p}}} 
\tikzset{decorate sep/.style 2 args=
{decorate,decoration={shape backgrounds,shape=circle,shape size=#1,shape sep=#2}}}

\title[\tiny{Holomorphic eta quotients of weight $1/2$}] 
{Holomorphic eta quotients of weight $1/2$}

\author{Soumya Bhattacharya}
\address 
{CIRM : FBK\\
via Sommarive 14\\
I-38123 Trento}

\email{soumya.bhattacharya@gmail.com}

\maketitle

 \begin{abstract}
  We 
  give 
  a short proof of Zagier's conjecture / Mersmann's theorem
  which states  
  that\hspace{1.3pt} 
  each holomorphic eta quotient of weight~1/2 
  is an integral rescaling of some eta quotient
  from Zagier's list of fourteen primitive
  holomorphic eta quotients. 
  In particular, given any holomorphic eta quotient $f$ of weight~1/2, 
   this result enables us to provide a closed-form expression for the coefficient of  $q^n$ in the $q$-series expansion of $f$, for all $n$.
   We also demonstrate 
   another
  application of the above theorem
  in extending the levels of the
  simple (resp. irreducible) holomorphic eta quotients.
 \end{abstract}
 \section{Introduction}
 The Dedekind eta function is defined by the infinite product:
 \begin{equation} \eta(z):=e^{\frac{\pi iz}{12}}\prod_{n=1}^\infty(1-e^{2\pi inz}) 
\label{17.4Aug}\end{equation} 
for all $z\in\H$, 
where 
$\H:=\{\tau\in\C\hs\idi\operatorname{Im}(\tau)>0\}$.
Eta is a holomorphic function on $\H$ with no zeros.
This function comes up naturally in many areas of Mathematics (see the Introduction in \cite{B-three} for a brief
overview of them). 
The function $\e$ is a modular form
of weight $1/2$ with a multiplier system 
on $\operatorname{SL}_2(\Z)$ (see \cite{b}).
An 
eta quotient $f$ is a finite product of the form 
\begin{equation}
 \prod\e_d^{X_d},
\label{13.04.2015}\end{equation}
where $d\in\N$, $\eta_d$ is the \emph{rescaling} of $\eta$ by $d$, defined by
\begin{equation}
 \e_d(z):=\e(dz) \ \text{ for all $z\in\H$}
\end{equation}
and the \emph{exponents} $X_d\in\Z$.
Eta quotients naturally inherit modularity 
from $\e$: The eta quotient $f$ in (\ref{13.04.2015}) transforms like a modular form of
weight $\frac12\sum_dX_d$ with a multiplier system on suitable congruence subgroups of $\operatorname{SL}_2(\Z)$: 
The largest among
these subgroups is 
\begin{equation}
 \Gm_0(N):=\Big{\{}\begin{pmatrix}a&b\\ c&d\end{pmatrix}\in
\operatorname{SL}_2(\Z)\hspace{3pt}\Big{|}\hspace{3pt} c\equiv0\hspace*{-.3cm}\pmod N\Big{\}},
\end{equation}
where 
\begin{equation}
 N:=\operatorname{lcm}\{d\in\N\hs\idi\hs X_d\neq0\}.
\end{equation}
We call $N$ the \emph{level} of $f$.
Since $\eta$ is non-zero on $\H$, 
the eta quotient $f$ 
is holomorphic if and only if $f$ does not have any pole at the cusps of 
$\Gamma_0(N)$.
We call an eta quotient $f$ \emph{primitive} if 
there does not exist any other eta quotient $h$ and any $\nu\in\N$
such that $f(z)=h(\nu z)$ for all $z\in\H$.

An \emph{eta quotient 
on $\Gm_0(M)$} is
an eta quotient whose level divides $M$.
Let $f$, $g$ and $h$ be nonconstant \hes on $\Gm_0(M)$
such that 
$f=g\times h$. Then we say that $f$ is \emph{factorizable on} $\Gm_0(M)$. 
We call a holomorphic eta quotient $f$ of level $N$ \emph{quasi-irreducible} (resp. \emph{irreducible}),
if it is not factorizable on $\Gm_0(N)$ (resp. on $\Gm_0(M)$ for all multiples~$M$ of $N$).
Here, it is worth mentioning that the notions of irreducibility and quasi-irreducibility of holomorphic eta quotients are conjecturally equivalent (see \cite{B-three}).
We say that a holomorphic eta quotient is \emph{simple} if it is 
 both primitive and quasi-irreducible.
 
 \section{Zagier's list} 
\label{17.5Sept}
Zagier 
observed 
that every holomorphic eta quotient of weight $1/2$ 
seems to originate through
integral rescalings of only a small number
of primitive eta quotients. 
He gave an explicit list (see \cite{z} or Theorem~\ref{m2} below) of fourteen primitive holomorphic eta quotients of weight $1/2$
and conjectured that the list is complete.
This conjecture was 
established by his student Mersmann in 
an 
excellent \emph{Diplomarbeit} \cite{c}. 
In his thesis, Mersmann also proved a more general conjecture of 
Zagier
which asserts that:~
\emph{There are only finitely many
simple holomorphic eta quotients of a given weight.} 
I provided
a much simplified proof of 
this conjecture in \cite{B-five}.
Furthermore in \cite{B-four}, I showed that the finiteness also holds
if we replace the word ``weight'' with ``level'' in the above conjecture.
In particular, since $1/2$ is the smallest possible weight of any holomorphic eta quotient,
no 
such eta quotient of weight~$1/2$ 
is a product of two
holomorphic eta quotients other than that of 1 and itself.
Thus, 
Mersmann's 
classification of 
holomorphic eta quotients of weight~$1/2$ is 
only a special
case of the last conjecture worked out in complete details:
\begin{theorem}[Classification 
of Holomorphic Eta Quotients of Weight 1/2]
Each holomorphic eta quotient of weight $1/2$ is 
a rescaling of one of the following eta quotients by a positive integer:
 \hypertarget{Zlist}{\textcolor{white}{\text{Zagier's list}}} 
$$\begin{array}{c}\eta\hspace{.7pt},\hspace{.7pt}\dfrac{\eta^2}{\eta_2}\hspace{.7pt}, \dfrac{\eta_2^2}{\eta}\hspace{.7pt},\hspace{.7pt}\dfrac{\eta_2^3}{\eta\hspace{.7pt}\eta_4}\hspace{.7pt},\hspace{.7pt}\dfrac{\eta_2^5}{\eta^2\hspace{.7pt}\eta_4^2}\hspace{.7pt}, \dfrac{\eta\hspace{.7pt}\eta_4}{\eta_2}\hspace{.7pt},\ \dfrac{\eta\hspace{.7pt}\eta_6^2}{\eta_2\eta_3}\hspace{.7pt}  
\hspace{.7pt},\hspace{.7pt}\dfrac{\eta^2\hspace{.7pt}\eta_6}{\eta_2\hspace{.7pt}\eta_3}\hspace{.7pt},\hspace{.7pt}\dfrac{\eta_2^2\hspace{.7pt}\eta_3}{\eta\hspace{.7pt}\eta_6}\hspace{.7pt},\hspace{.7pt}\dfrac{\eta_2\hspace{.7pt}\eta_3^2}{\eta\hspace{.7pt}\eta_6}\\
\\
\dfrac{\eta_2^2\hspace{.7pt}\eta_3\hspace{.7pt}\eta_{12}}{\eta\hspace{.7pt}\eta_4\hspace{.7pt}\eta_6}\hspace{.7pt},\hspace{.7pt} \dfrac{\eta_2^5\hspace{.7pt}\eta_3\hspace{.7pt}\eta_{12}}{\eta^2\hspace{.7pt}\eta_4^2\hspace{.7pt}\eta_6^2}
\hspace{.7pt},\hspace{.7pt} \dfrac{\eta\hspace{.7pt}\eta_4\hspace{.7pt}\eta_6^2}{\eta_2\hspace{.7pt}\eta_3\hspace{.7pt}\eta_{12}}\hspace{.7pt},\hspace{.7pt} \dfrac{\eta\hspace{.7pt}\eta_4\hspace{.7pt}\eta_6^5}{\eta_2^2\hspace{.7pt}\eta_3^2\hspace{.7pt}\eta_{12}^2}.
\end{array}$$
\label{m2}
\end{theorem}
Though Mersmann's proof of the above theorem 
is indeed ingenious (see \cite{c}), but K\"ohler
at p.~117 in \cite{b} also complains about its length and lack of lucidity.
So, after briefly discussing some applications of Theorem~\ref{m2} in the
next two sections, 
we shall
present a shorter and simpler proof of the theorem. 
Except in a few places (for example, see the proof of Lemma~\ref{tlem}),
we shall closely follow the basic ideas in Mersmann's proof.
\section{$q$-series expansions of the 
eta quotients in Zagier's list}
Theorem~\ref{m2} implies that in order to obtain the $q$-series expansion of 
any holomorphic eta quotient of weight $1/2$, it suffices only to know the $q$-series expansions of the 
eta quotients in \hyperlink{Zlist}{{\text{Zagier's list}}}. Recall that the Jacobi triple product identity states: 
\begin{equation}
\prod_{n=1}^\infty(1-x^{2n})(1+x^{2n-1}y)(1+x^{2n-1}y^{-1})=\sum_{n=-\infty}^\infty x^{n^2}y^n
\label{jacpro}\end{equation}
for all $x,y\in\C$ such that $|x|<1$ and $y\neq0$\hspace{1pt} (see Theorem~352 in \cite{hw}).
Suitable substitutions in the above identity 
shows that each 
of the eta quotients in \hyperlink{Zlist}{{\text{Zagier's list}}} 
has a theta series representation
with fractional exponents 
(for a list of the 
required substitutions, see Table~\ref{tab} below or Remark~4.9 in \cite{B-two}).
These theta series representations 
of the eta quotients in 
\hyperlink{Zlist}{{\text{Zagier's list}}} 
are not much of a surprise
since 
Serre-Stark theorem \cite{st} implies that all
suitable integral rescalings of 
holomorphic eta quotients of weight $1/2$
are linear combinations of theta series of the form 
\begin{equation}
\displaystyle{\sum_{n\in\Z}\psi(n)q^{tn^2}},
\label{23-06}\end{equation}
where $\psi$ is an even Dirichlet character, $t\in\N$ and\hypertarget{queue}{}
$$q^r=q^r(z):=e^{2\pi irz}\hspace{1pt}\text{ for all $r$.}$$

\begin{center}
\begin{table}[h]
\caption{Zagier's list via substitutions in the Jacobi triple product}
\begin{tabular}{l|cccccccc}
\diagbox[width=3.71em]{\hspace*{.272cm}$x$}{\\$y$}&\hspace*{5pt}$q^{1/2}$&$iq^{1/2}$&$-q^{1/2}$&$-iq^{1/2}$&$\omega q^{1/2}$&$i\omega q^{1/2}$&$-q^{2}$&$-\omega q^{2}$\\\hline
\\
\hspace*{.272cm}$q^{1/2}$&$\dfrac{\eta_2^2}{\eta}$&&&&$\dfrac{\eta^2\hspace{.7pt}\eta_6}{\eta_2\hspace{.7pt}\eta_3}$\\ 
\hspace*{.082cm}$iq^{1/2}$&&$\dfrac{\eta\hspace{.7pt}\eta_4}{\eta_2}$&&&&$\dfrac{\eta_2^5\hspace{.7pt}\eta_3\hspace{.7pt}\eta_{12}}{\eta^2\hspace{.7pt}\eta_4^2\hspace{.7pt}\eta_6^2}$\\ 
\hspace*{.272cm}$q$&&&&&&&$\dfrac{\eta^2}{\eta_2}$&$\dfrac{\eta_2^2\hspace{.7pt}\eta_3}{\eta\hspace{.7pt}\eta_6}$\\ \\
$-q$&&&&&&&$\dfrac{\eta_2^5}{\eta^2\hspace{.7pt}\eta_4^2}$&$\dfrac{\eta\hspace{.7pt}\eta_4\hspace{.7pt}\eta_6^2}{\eta_2\hspace{.7pt}\eta_3\hspace{.7pt}\eta_{12}}$\\
\hspace*{.272cm}$q^{3/2}$&$\dfrac{\eta_2\hspace{.7pt}\eta_3^2}{\eta\hspace{.7pt}\eta_6}$&&$\e$\\ 
\hspace*{.142cm}$iq^{3/2}$&&$\dfrac{\eta_2^3}{\eta\hspace{.7pt}\eta_4}$&&$\dfrac{\eta\hspace{.7pt}\eta_4\hspace{.7pt}\eta_6^5}{\eta_2^2\hspace{.7pt}\eta_3^2\hspace{.7pt}\eta_{12}^2}$\\
\hspace*{.272cm}$q^{3}$&&&&&&&$\dfrac{\eta\hspace{.7pt}\eta_6^2}{\eta_2\eta_3}$\\ \\
$-q^{3}$&&&&&&&$\dfrac{\eta_2^2\hspace{.7pt}\eta_3\hspace{.7pt}\eta_{12}}{\eta\hspace{.7pt}\eta_4\hspace{.7pt}\eta_6}$
\end{tabular}
\label{tab}
\end{table}
\end{center}

In the above table, $\omega$ denotes a primitive cube root of unity.
Making the substitutions in $(\ref{jacpro})$ as detailed in the Table~\ref{tab}, we see that for
each eta quotient $f$ in \hyperlink{Zlist}{{\text{Zagier's list}}},
there exists a $t\idi 24$ and a sequence $\{a_n\}_n$ of complex numbers such that 
\begin{equation}
f(z)=\sum_{n}a_nq^{tn^2/24}.
\label{thetaseries}\end{equation}
(see also \cite{aok}, \cite{z}, \cite{z1}, \cite{z2} or Chapter~8 in \cite{b}).
In particular, identifying 
the eta quotients in \hyperlink{Zlist}{{\text{Zagier's list}}} whose corresponding sequences $\{a_n\}_n$ define Dirichlet characters,
we obtain 
an elementary proof of Theorem~1.1.1 in \cite{jlo}.

Also, from Table~\ref{tab}, (\ref{jacpro}) and (\ref{17.4Aug}),
we see
that the translation $z\mapsto z+\frac12$ or equivalently 
the \emph{sign transform} $q\mapsto-q$ 
induces an involution of  \hyperlink{Zlist}{{\text{Zagier's}}} 
\hyperlink{Zlist}{{\text{list}}} 
of eta quotients up to 
multiplication by a 48-th root of unity (see \cite{b}, \cite{c}, \cite{z1} and \cite{z2}).

\section{Extending the levels of simple holomorphic eta~quotients}

Recall that a simple holomorphic eta quotient is both primitive and quasi-irreducible.
In particular, since primitivity and simplicity are synonymous for holomorphic eta quotients of weight~$1/2$,
\hyperlink{Zlist}{{\text{Zagier's list}}} supplies us with examples of simple
holomorphic eta quotients of levels $1, 2, 4, 6$ and $12$.
Also,
Corollary~2 in \cite{B-three} implies that there are simple holomorphic eta quotients of all
prime levels. Thus, we see that
there exist simple holomorphic eta quotients
of all levels less than 8. It is easy to check that, any quasi-irreducible holomorphic eta quotient of level 8 is 
a rescaling of some eta quotient of
level 1,2 or 4. In other words, there does not exist any simple holomorphic eta quotient of level 8.
It is still an open problem to classify the levels of~which 
simple holomorphic eta quotients exist. 
However, Theorem~\ref{m2}
implies that given a simple (resp. irreducible) holomorphic eta quotient $f$ of 
a suitable level, we can construct up to thirteen new simple (resp. irreducible)
holomorphic eta quotients of the same weight as of $f$ 
but of higher levels:
\begin{co}
Let there be a simple $($resp. irreducible$)$ holomorphic eta quotient of an odd level~$N$.
Then there are at least two simple $($resp. irreducible$)$ holomorphic eta quotients of level~$2N$ and 
and at least three simple $($resp. irreducible$)$ holomorphic eta quotients of level~$4N$. Moreover, if 
$3\ndi N$, then there are also four simple $($resp. irreducible$)$ holomorphic eta quotients of levels~$6N$ and 
four simple $($resp. irreducible$)$ holomorphic eta quotients of~$12N$. 
\label{21/06}
\end{co}
\begin{proof}
 Let $f$ be a 
 \he of 
 level $N$. Let $g$ be a 
 \he of weight~1/2 from
 \hyperlink{Zlist}{\text{Zagier's list}} (see Theorem~\ref{m2})
 such that $N$ is coprime to the level of $g$. 
 Since $g$ is primitive, it follows that $f$ is primitive if and only if so is
 \begin{equation}
 f\circledast g:=\prod_{d\idi M}\prod_{d'\idi N}\e_{dd'}^{X_{d}Y_{d'}},
 \end{equation} 
where $f=\prod_{d\idi M}\e^{X_{d}}$ and $g=\prod_{d'\idi N}\e^{Y_{d'}}$.
 Again, since $g$ is of weight~1/2, 
 from Lemma~2 and Corollary~5 in \cite{B-three},
 it follows that $f\circledast g$ is quasi-irreducible (resp. irreducible) 
 if and only if so~is~$f$.
\end{proof}

Section~6 in \cite{B-four} furnishes several examples of irreducible holomorphic
eta quotients. In particular, if $N$ is cubefree (i.~e., if $N$ is not divisible by the cube of any integer except 1), then the holomorphic eta quotient of level~$N$ 
constructed in the 
proof of Theorem~3 in \cite{B-four}\hspace{2pt} is in particular, simple
(see also Theorem~6.2 and Corollary~6.3 in \cite{B-two}). 
We shall also supply an infinite family of simple holomorphic eta quotients in \cite{B-seven} whose levels are not cubefree.

\section{Notations and the basic facts}

By $\N$ we denote the set of positive integers.
For $N\in\N$, by $\D$ we denote the set of divisors of $N$.
For 
$X\in\ZD$, we define the eta quotient $\e^X$
   by
   \begin{equation}\label{3Jan15.1}
    \e^X:=\displaystyle{\prod_{d\in\D}\eta_d^{X_d}},
 \end{equation}
where $X_d$ is the 
value of $X$ 
at $d\in\D$ whereas $\e_d$ denotes the rescaling of $\e$ by $d$.
Clearly, the level of $\e^X$ divides $N$. In other words, $\e^X$ transforms like a modular form on $\GN$. 
We define the summatory function $\sigma:\ZD\rightarrow\Z$ by \begin{equation}
\sigma(X):=\sum_{d\in\D}X_d.
\label{30.08.2015.A}\end{equation}
Since $\e$ is of weight $1/2$,  
the weight of $\e^X$ is  $\sigma(X)/2$ for all $X\in\ZD$.

An \emph{eta quotient 
on $\GN$} is
an eta quotient whose level divides $N$.
We recall 
that 
such an eta quotient $f$ is holomorphic if it
does not have any poles at the cusps of $\GN$. Under the action of $\GN$ on $\PQ$
by M\"obius transformation, for 
$a,b\in\Z$ with $\gcd(a,b)=1$,
we have
\begin{equation}
[a:b]\hspace{.1cm}
{{{{\sim}}}}_{\hspace*{-.05cm}{{{{\GN}}}}}
\hspace{.08cm}[a':\gcd(N,b)]\label{19.05.2015} 
\end{equation}
for some $a'\in\Z$ which is coprime to $\gcd(N,b)$ (see \cite{ds}).
We identify $\PQ$ with $\Q\cup\{\infty\}$ via the canonical bijection that maps $[\alpha:\lambda]$ to 
$\alpha/\lambda$ if $\lambda\neq0$ and to $\infty$ if $\lambda=0$. 
For $s\in\Q\cup\{\infty\}$ and a weakly holomorphic modular form $f$ on $\GN$, the order of $f$ at the cusp $s$ of $\GN$ is 
the exponent of 
 \hyperlink{queue}{$q^{{1}/{w_s}}$} 
occurring with the first nonzero coefficient in the 
$q$-expansion of $f$
at the cusp $s$,
where $w_s$ is the width of the cusp $s$ (see \cite{ds}, \cite{ra}).
The following is a minimal set of representatives of the cusps of $\GN$ (see \cite{ds}, \cite{ymart}):
\begin{equation}
\cl{S}_N:=\Big{\{}\frac{a}{t}\in\Q\hspace{2.5pt}{\di}\hspace{2.5pt}t\in\cl{D}_N,\hspace{2pt} a\in\bb{Z}, 
 \hspace{2pt}\gcd(a,t)=1\Big{\}}/\sim\hspace{1.5pt},
\end{equation}
where $\dfrac{a}{t}\sim\dfrac{b}{t}$ if and only if $a\equiv b\pmod{\gcd(t,N/t)}$.
For $d\in\D$ and for $s=\dfrac{a}t\in\cl{S}_N$ with $\gcd(a,t)=1$, we have
\begin{equation}
 \ord_s(\e_d\hspace{1pt};\GN)= \frac{N\cdot\gcd(d,t)^2}{24\cdot d\cdot\gcd(t^2,N)}\in\frac1{24}\N 
\label{26.04.2015}\end{equation}
(see 
\cite{ymart}). 
It is easy to check the above inclusion  
when $N$ is a prime power. 
The 
case 
for general $N$ 
follows by multiplicativity (see (\ref{27.04.2015}), (\ref{20.3Sept}) and (\ref{13May})).
It follows that for all $X\in\ZD$, we have
\begin{equation}
  \ord_s(\e^X\hspace{1pt};\GN)= \frac1{24}\sum_{d\in\D}\frac{N\cdot\gcd(d,t)^2}{d\cdot\gcd(t^2,N)}X_d\hspace{1.5pt}. 
\label{27.04}\end{equation}
 In particular, 
that implies
\begin{equation}
 \ord_{a/t}(\e^X\hspace{1pt};\GN)=\ord_{1/t}(\e^X\hspace{1pt};\GN)
\label{27.04.2015.1}\end{equation}
for all $t\in\D$ and for all the $\varphi(\gcd(t,N/t))$ inequivalent cusps of $\GN$
represented by rational numbers
of the form $\dfrac{a}{t}\in\cl{S}_N$ with $\gcd(a,t)=1$,
where $\varphi$ denotes Euler's totient function.

We define the \emph{\hypertarget{om}{order map}}
$\cl{O}_N:\ZD\rightarrow\frac1{24}\ZD$ of level $N$ as the map which sends $X\in\ZD$ to the ordered set of
orders of the eta quotient $\e^X$ at the cusps $\{1/t\}_{t\in\D}$ of $\GN$. Also, we define
\emph
{order matrix} $A_N\in\Z^{\D\times\D}$ of level $N$ by 
\begin{equation}
 A_N(t,d):=24\cdot\ord_{1/t}(\e_d\hspace{1pt};\GN) 
\label{27.04.2015}\end{equation}
for all $t,d\in\D$.
By linearity of the order map, we have 
\begin{equation}
\cl{O}_N(X)=\frac1{24}\cdot A_NX\hspace{1.5pt}. 
\label{28.04}\end{equation}
From (\ref{27.04.2015}) and (\ref{26.04.2015}), we note that the matrix 
$A_N$ is not symmetric. 
It would have been much easier for us to work with $A_N$
if it would have been 
symmetric 
(for example, see
(\ref{08.10.2015})).
 So, we define the \emph{symmetrized order matrix} $\A_N\in\Z^{\D\times\D}$
by  
\begin{equation}
 \widehat{A}_N(t,\un)=\gcd(t,N/t)\cdot A_N(t,\un)\hspace{6pt}\text{for all $t\in\D$},
\label{20.3Sept}\end{equation}
where $\A_N(t, \un)$ (resp. $A_N(t, \un)$) denotes the row of $\A_N$ (resp. $A_N$) indexed by $t\in\D$.
For example, for a prime power $p^n$, 
we have
\begin{equation}
\A_{p^n}=\begin{pmatrix}
 \vspace{5.8pt}p^n &p^{n-1} &p^{n-2} 
&\cdots  &p &1\\
\vspace{5.8pt} p^{n-1} &p^n &p^{n-1} 
&\cdots  &p^2 &p\\
 \vspace{5.8pt}p^{n-2} &p^{n-1} &p^n 
&\cdots  &p^3 &p^2\\
\vspace{5.8pt} \vdots &\vdots &\vdots 
&\cdots &\vdots &\vdots\\
\vspace{5.8pt} p &p^2 &p^3 
&\cdots &p^n &p^{n-1}\\
 1 &p &p^2 
&\cdots  &p^{n-1} &p^n
\end{pmatrix}.
 \label{23July}
\end{equation}
For $r\in\N$, if $Y,Y'\in\Z^{\D^{\hspace{.5pt}r}}$ is 
such that $Y-Y'$ is
nonnegative (resp. positive) at each element of $\D^{\hspace{.5pt}r}$, then
 we write $Y\geq Y'$ (resp. $Y>Y'$). 
In particular, for $X\in\ZD$, the eta quotient $\e^X$ is holomorphic if and only if
$\A_NX\geq0$.
From $(\ref{20.3Sept})$,
$(\ref{27.04.2015})$ and $(\ref{26.04.2015})$, we note that $\A_N(t,d)$ is multiplicative in $N$ and in $d,t\in\D$.
Hence, it follows that
\begin{equation}
 \A_N=\bigotimes_{\substack{p^n\|N\\\text{$p$ prime}}}\A_{p^n},
\label{13May}\end{equation}
where by \hspace{1.5pt}$\otimes$\hspace{.2pt},  we denote the Kronecker product of matrices.\footnote{Kronecker product of matrices is
not commutative. However, 
since any given ordering of the primes dividing $N$ induces a lexicographic ordering on $\cl{D}_N$ 
with which the entries of $\A_N$ are indexed, 
Equation (\ref{13May}) makes sense for all possible 
orderings of the primes dividing $N$.} 

It is easy to verify that for a prime power $p^n$, 
the matrix $\A_{p^n}$ is invertible with the tridiagonal inverse: 
\begin{equation}
\A_{p^n}^{-1}=
\frac{1}{p^n(1-\frac{1}{p^2})}
\begin{pmatrix}
\hspace{6pt}1 &\hspace*{-6pt}-\frac{1}{p} &  
& &  & \\
\vspace{5pt}\hspace*{-1pt}-\frac{1}{p} & \hspace*{-2pt}1+\frac{1}{p^2} &\hspace*{-6pt}-\frac{1}{p} 
& &\textnormal{\Huge 0} & \\
\vspace{5pt}&\hspace*{-6pt}-\frac{1}{p} & \hspace*{-4pt}1+\frac{1}{p^2} &\hspace*{-6pt}-\frac{1}{p} 
 &  &  \\
 &  & \ddots 
 & \ddots & \ddots & \\
\vspace{5pt}\hspace{2pt} &\textnormal{\Huge 0}  &  
 &\hspace*{-6pt}-\frac{1}{p} & \hspace*{-4pt}1+\frac{1}{p^2} &\hspace*{-6pt}-\frac{1}{p}\hspace{2pt}\\
\hspace{2pt} &  & 
&  &\hspace*{-6pt}-\frac{1}{p} & 1\hspace{2pt}
\end{pmatrix}.
 \label{r1}\end{equation}
For general $N$, the invertibility of the matrix $\A_N$ now 
follows by (\ref{13May}).
Hence, any eta quotient on $\GN$ is uniquely determined by its orders at the set of the cusps
$\{1/t\}_{t\in\D}$ of $\GN$. In particular, for distinct $X,X'\in\ZD$, we have $\e^X\neq\e^{X'}$. The last statement 
is also implied by the uniqueness of $q$-series expansion:
Let $\e^{\widehat{X}}$ and $\e^{\widehat{X}'}$
be the \emph{eta products} (i.~e.  $\widehat{X}, \widehat{X}'\geq0$)
obtained by multiplying $\e^X$ and $\e^{X'}$ with a common denominator. The claim follows by induction on the weight of $\e^{\widehat{X}}$
(or equivalently, the weight of $\e^{\widehat{X}'}$)
when we compare
the corresponding 
first two exponents of $q$
occurring in the $q$-series expansions of 
$\e^{\widehat{X}}$ and $\e^{\widehat{X}'}$.
 
\section{A holomorphy preserving map}
\label{mapsec}
For $d\in\D$, we say that $d$ \emph{exactly divides} $N$ (and write $d\|N$)
if 
$d$ and $N/d$ are mutually coprime. For any $d\|N$ , there exists a canonical bijection 
between $\ZD$ and  $\bb{Z}^{\cl{D}_{{N/d}}\times\cl{D}_d}$. 
We denote by $X^ {{[d]}}$ the 
image
of $X\in\bb{Z}^{\D}$ in $\bb{Z}^{\cl{D}_{{N/d}}\times\cl{D}_d}$ under this bijection.
From the facts that $\A_{{N}}=\A_{{d}}\otimes\A_{{N/d}}$
and that these matrices are symmetric, it follows that
\begin{equation}
(\A_{{N}}X)^{{[d]}}=\A_{{N/d}}X^{{[d]}}\A_{{d}}
\label{08.10.2015}\end{equation}
(see Lemma 4.3.1 in \cite{a}). 

Now we 
provide a holomorphy preserving map which will be very useful in the next two sections.
Let $M,N\in\N$ with $N\|M$ and let $\{a_n\}_{n\in\N}$ be a multiplicative sequence of 
 integers such that
$a_n=0$ if
 $n\notin\cl{D}_{M/N}$
and 
\begin{equation}
a_{p^{j-1}}+ a_{p^{j+1}}\leq\begin{cases}
pa_{p^{j}}&\text{if $p^j\|{(M/N)}$}\\
(p+\frac1p)a_{p^{j}}&\text{otherwise}
\end{cases}
\label{04.11.2015.4}
\end{equation}
for all primes $p$
and for all nonnegative integers $j$ such that 
$p^j\idi
{(M/N)}$. Here, we set $a_{p^{-1}}:=0$ for each prime $p$.
We define 
$\widehat{a}\in\Z^{\cl{D}_{M/N}}
$ 
by 
$\widehat{a}_d:=a_d$ for all $d\in\cl{D}_{M/N}$.
\begin{lem}
Let $M,N$, $\{a_n\}_{n\in\N}$ and\hspace*{1.5pt} ${\widehat{a}}$ be as above.
Then the homomorphism \begin{equation}{\varPhi}_{{M,N,}\widehat{a}}(\e^X):= 
\e^{
X^{{[M/N]}}\hspace*{1.5pt}\cdot\hspace*{1.5pt}{\widehat{a}}}
\label{07.11.2015}\end{equation}
from the multiplicative group of eta quotients on $\Gamma_0(M)$ to that on $\GN$ preserves holomorphy.
Moreover, if  $\widehat{a}$ is such that the strict inequality in $(\ref{04.11.2015.4})$
holds for all prime powers $p^j$ which divide $M/N$, then $\varPhi_{{M,N,}\widehat{a}}$ does not map 
any nonconstant holomorphic eta quotient to $1$. 
\label{04.11.2015.5}\end{lem}
\begin{proof}Let $X\in\Z^{\cl{D}_M}$ be such that the eta quotient $\e^X$ is holomorphic.
Since $\e^X$ is holomorphic if and only if 
\begin{equation}
\A_{{M}}X\geq0,
\label{04.11.2015}\end{equation}
we only require
to show that 
\begin{equation}
\A_{{N}}X^{{[M/N]}}\hspace{1.2pt}{\widehat{a}}\geq0.
\label{04.11.2015.1}\end{equation}
From (\ref{08.10.2015}), we get: 
\begin{equation}
(\A_{{M}}X)^{{[M/N]}}\A^{-1}_{{M/N}}\hspace*{1pt}{\widehat{a}}=\A_{{N}}X^{{[M/N]}}\hspace{1.2pt}{\widehat{a}}\hspace{1.2pt},
\label{someday}\end{equation}
It follows from the last equality and (\ref{04.11.2015}) that it is enough to show: 
\begin{equation}
 \A^{-1}_{{M/N}}\hspace*{1pt}{\widehat{a}}\geq0.
\label{04.11.2015.2}\end{equation}
From 
multiplicativity of the sequence $\{a_n\}_{n\in\N}$, we get that
 \begin{equation}
{\widehat{a}}=\bigotimes_{\substack{p^n\|(M/N)\\\text{$p$ prime}}}{\widehat{a}}^{\hspace*{1pt}(p^n)},
\label{04.11.2015.3}\end{equation}
where ${\widehat{a}}^{\hspace*{1pt}(p^n)}\in\Z^\cl{D}_{p^n}$ is defined by ${\widehat{a}}^{\hspace*{1pt}(p^n)}_{p^j}:=a_{p^j}$ for all $j$
such that 
$p^j\idi
{(M/N)}$. 
Now, (\ref{13May}) and (\ref{04.11.2015.3}) together imply that
for (\ref{04.11.2015.2}) to hold, it suffices if
so does the following inequality:
\begin{equation}
 \A^{-1}_{p^n}\hspace*{1pt}{\widehat{a}}^{\hspace*{1pt}(p^n)}\geq0
\label{04.11.2015.6}\end{equation}
 for each prime power $p^n\|(M/N)$. 
The 
last  
inequality 
follows immediately from (\ref{r1}) and 
(\ref{04.11.2015.4}). 

The strict inequality in $(\ref{04.11.2015.4})$
implies the strict inequality in (\ref{04.11.2015.6}) which in turn,
implies the strict inequality in (\ref{04.11.2015.2}).
Since $\e^X$ is  holomorphic, it has nonnegative order of vanishing
at all cusps of $\Gm_0(M)$. Moreover, 
if $\e^X$ is nonconstant, then its order of vanishing
at some cusp of $\Gm_0(M)$ is nonzero.
So, (\ref{28.04}) implies that all the entries of $A_MX$ are nonnegative
and at least one of its entries is nonzero. Hence via (\ref{someday}),
we conclude that 
 $\A_{{N}}X^{{[M/N]}}\hspace{1.2pt}{\widehat{a}}\neq0$. 
\end{proof}
\begin{co}
Let $M\in\N$ and let
$p$ be a prime divisor of $M$. Let $M=p^nN$, where $p\nidi N$.
For some $j\in\{0,1,\hdots,n\}$ and for some $m\in\{1,\hdots,p-1\}$, define ${\widehat{a}}={\widehat{a}}_{{(j,m)}} %
\in\Z^{\cl{D}_{p^n}}$
by \label{07.11.2015.1}
\begin{equation}
 {\widehat{a}}(p^i)=\begin{cases}
                        m&\text{if $i=j$}\\
                        1&\text{otherwise.}\\
                       \end{cases}
\end{equation}
Then the homomorphism ${\varPhi}_{{M,N,}\widehat{a}}$ preserves holomorphy $($see $(\ref{07.11.2015}))$ and it does not 
map any nonconstant holomorphic eta quotient to $1$.
\end{co}

\begin{co}
 Let $M$, $N$ and $p\geq3$ be as in Corollary~$\ref{07.11.2015.1}$.
 Define $\mathds{1}_{N}\in\Z^{\cl{D}_{N}}$ by\label{07.11.2015.2}
 \begin{equation}
\mathds{1}_N(d):=1 
\text{ for all \hspace{1pt}$d\in\cl{D}_N$},
\end{equation}
 Let $X\in\Z^{\cl{D}_M}$ be such that $\e^X$ is a holomorphic 
 eta quotient of weight $1/2$ on $\Gm_0(M)$. Then
 ${\varPhi}_{{M,p^n,{\mathds{1}_{N}}}}(\e^X)=\e_{p^{{j_{{0}}}}}$ for some $j_0\in\{0,1,\hdots,n\}$.
\end{co}
\begin{proof}
Let $j\leq n$ be a nonnegative integer and define ${\widehat{a}}
\in\Z^{\cl{D}_{p^n}}$
by
\begin{equation}
 {\widehat{a}}(p^i)=\begin{cases}
                        p-1&\text{if $i=j$}\\
                        1&\text{otherwise.}\\
                       \end{cases}
\label{09.11.2015}\end{equation}
Let $Y\in\Z^{\D}$ be such that $\e^Y={\varPhi}_{{M,N,}\widehat{a}}(\e^X)$.
Then Corollary~\ref{07.11.2015.1} implies that $\e^Y$ is holomorphic.
Let $Z\in\Z^{\cl{D}_{p^n}}$
 be such that $\e^Z={\varPhi}_{{M,p^n,{\mathds{1}_{N}}}}(\e^X)$.
Then from (\ref{30.08.2015.A}), (\ref{07.11.2015}) and (\ref{09.11.2015}), it follows that
\begin{equation}
\sigma(Y)=\sigma(X)+(p-2)Z_{p^j}=1+(p-2)Z_{p^j},
\end{equation}
where the last equality holds since $\e^X$ is of weight $1/2$. 
Since $p\geq3$, we have $\sigma(Y)\geq0$ if and only if $Z_{p^j}\geq0$.
Since $\e^Y$ is holomorphic, $\sigma(Y)$ is indeed nonnegative and 
so is $Z_{p^j}$ for all $j\in\{0,1,\hdots,n\}$ .
Since the map ${\varPhi}_{{M,p^n,{\mathds{1}_{N}}}}$
preserves weight, we have
$$\sigma(Z)=\sigma(X)=1.$$
Hence, there exists $j_0\in\{0,1,\hdots,n\}$ such that
\begin{equation}
 Z_{p^j}=\begin{cases}
                        1&\text{if $j=j_0$}\\
                        0&\text{otherwise.}\\
                       \end{cases}
\end{equation}
In other words, we have $\e^Z=\e_{p^{{j_{{0}}}}}$.
\end{proof}

In particular, if we set $M=p^n$ in Corollary~\ref{07.11.2015.2}, then ${\varPhi}_{{M,p^n,{\mathds{1}_{N}}}}$ becomes the identity map on 
eta quotients on $\Gm_0(M)$. So,  Corollary~\ref{07.11.2015.2} implies that 
\begin{co}
For all primes $p\geq3$ and for all 
$n\in\N$,
the only holomorphic eta quotient of weight $1/2$ and level $p^n$ is $\e_{p^n}$.
\label{19.11.2015.3}\end{co}

\section{Reduction 
to 3-smooth levels}
\label{indsec}
For $m\in\N$, 
an integer $N$ is called $m$-smooth if none of the prime factors of $N$ 
is greater than $m$. We have
\begin{lem}
If for all $3$-smooth 
$N\in\N$, the only primitive holomorphic eta quotients on $\GN$ 
are 
those given in 
\hyperlink{Zlist}{{\text{Zagier's list}}} $($see Theorem~\ref{m2}$)$, then the same is true 
for all $N\in\N$.
\end{lem}
\begin{proof}
We shall proceed by induction on the greatest prime divisor of $N$:
Let us assume that
for some $m\geq3$, the only primitive holomorphic eta quotients on $\GN$ 
for each $m$-smooth $N\in\N$\hspace*{1.1pt}
are
those given in \hyperlink{Zlist}{{\text{Zagier's list}}}.

Suppose, there exists a primitive holomorphic eta quotient of some level $M=p^nN$
with $p>m$, where $N$ is $m$-smooth. Let $Y\in\Z^{\cl{D}_{p^n}}$ be such that
$\e^Y={\varPhi}_{{M,p^n,{\mathds{1}_{N}}}}(\e^X)$. 
From Corollary~\ref{07.11.2015.2}, we know that there exists a $j_0\in\{0,1,\hdots,n\}$
such that
$\e^Y=\e_{p^{{j_{{0}}}}}$. Let $i_0\in\{0,n\}$ be distinct from $j_0$.
Then we have $Y_{p^{i_0}}=0$. 
Define ${\widehat{a}}
\in\Z^{\cl{D}_{p^n}}$
by
\begin{equation}
 {\widehat{a}}(p^i)=\begin{cases}
                        4&\text{if $i=i_0$}\\
                        1&\text{otherwise.}\\
                       \end{cases}
\label{10.11.2015}\end{equation}
Since $p\geq5$, Corollary~\ref{07.11.2015.1} implies that
both $
{\varPhi}_{{M,N,}\widehat{a}}(\e^X)$ and
$
{\varPhi}_{{M,N,{\mathds{1}_{p^n}}}}(\e^X)$
are holomorphic. Let $Z\in\ZD$ such that $\e^Z={\varPhi}_{{M,N,}\widehat{a}}(\e^X)$.
Then from (\ref{30.08.2015.A}), (\ref{07.11.2015}) and (\ref{10.11.2015}), it follows that
\begin{equation}
\sigma(Z)=\sigma(X)+3Y_{p^{i_0}}=1,
\end{equation}
where the last equality holds since $\e^X$ is of weight $1/2$ and since $Y_{p^{i_0}}=~0$.
Hence, the eta quotient $f:=\e^Z$ is of weight $1/2$. Again, since the map $
{\varPhi}_{{M,N,{\mathds{1}_{p^n}}}}$
preserves weight, the eta quotient $g:=
{\varPhi}_{{M,N,{\mathds{1}_{p^n}}}}(\e^X)$
is also of weight~$1/2$.
Since $\e^X$ is a primitive eta quotient of level $M$,
there exists some $d\idi N$ such that $X^{{[p^n]}}_{d, p^{i_0}}\neq0$.
So from (\ref{07.11.2015}) and (\ref{10.11.2015}), it follows that
${\varPhi}_{{M,N,}\widehat{a}}(\e^X)\neq{\varPhi}_{{M,N,{\mathds{1}_{p^n}}}}(\e^X)$,
i.~e. $f\neq g$.
Now, by the induction hypothesis, 
we have:
$f=f'_{d_1}$ and $g=g'_{d_2}$, where $f'$ and $g'$ 
belong to 
\hyperlink{Zlist}{{\text{Zagier's list}}} and $d_1,d_2\in\N$.
Here, $f'_{d_1}$ (resp. $g'_{d_2}$) denotes the rescaling of $f'$ by $d_1$ (resp. the rescaling of $g'$ by $d_2$).
From (\ref{10.11.2015}), we see that the corresponding exponents of $f$ and $g$ are congruent modulo~3. 
Hence from 
\hyperlink{Zlist}{{\text{Zagier's list}}} , it follows that $d_1=d_2$ and that the set $\{f',g'\}$ is either 
$\Big{\{}\dfrac{\eta^2}{\eta_2}\hspace{.7pt}, \dfrac{\eta_2^2}{\eta}\Big{\}}$
or 
$\Big{\{}\dfrac{\eta_2^5}{\eta^2\hspace{.7pt}\eta_4^2}\hspace{.7pt}, \dfrac{\eta\hspace{.7pt}\eta_4}{\eta_2}\Big{\}}$.
Both of these possibilities imply 
that the eta quotient ${\varPhi}_{{M,N,}\widehat{a}'}(\e^X)$ 
is not holomorphic, where 
${\widehat{a}'}
\in\Z^{\cl{D}_{p^n}}$
is defined by
\begin{equation}
 {\widehat{a}'}(p^i)=\begin{cases}
                        2&\text{if $i=i_0$}\\
                        1&\text{otherwise.}\\
                       \end{cases}
\label{10.11.2015.1}\end{equation}
Thus, we get a contradiction to Corollary~\ref{07.11.2015.1}\hspace{1.8pt}! 
\end{proof}
\section{The 
completeness of 
Zagier's list}

In the following, we shall 
show that each holomorphic eta quotient of weight~$1/2$ and of a $3$-smooth
level is a rescaling of some eta quotient on $\Gm_0(72)$  by a positive integer.
By using 
standard linear~algebraic 
techniques (for 
example, see Chapter~4 in \cite{b}), it is 
easy to check 
that the only primitive holomorphic 
eta quotients of weight $1/2$ on $\Gm_0(72)$ 
are those given in 
\hyperlink{Zlist}{{\text{Zagier's list}}}.
Thus, we 
obtain the completeness of
the \hyperlink{Zlist}{{\text{list}}}.

\begin{lem}
For an integer $n>2$ and for $X\in\Z^{D_{2^n}}$, if the eta quotient $\e^X$ is holomorphic, then
we have\label{l2}
\begin{equation}
|X_1|+|X_{2^n}|\leq2\cdot\sigma(X).
\label{10.11.2015.2}\end{equation}
Moreover, if 
equality holds in $(\ref{10.11.2015.2})$, then 
both $X_1$ and $X_{2^n}$ are even.
\end{lem}
\begin{proof}
Let $X\in\Z^{\cl{D}_{2^n}}$ and define ${\widehat{a}}
\in\Z^{\cl{D}_{2^n}}$
by
\begin{equation}
 {\widehat{a}}(2^j)=\begin{cases}
                        2-\operatorname{sgn}(X_{2^j})&\text{if $j=0$ or $j=n$.}\\
                        2&\text{otherwise.}\\
                       \end{cases}
\label{10.11.2015.A}\end{equation}
Then from Lemma \ref{04.11.2015.5}, it follows that 
${\varPhi}_{{2^n,1,}\hspace{.5pt}\widehat{a}}(\e^X)=\e^{2\cdot\sigma(X)-|X_1|-|X_{2^n}|}$ is holomorphic.
So, we have
$$|X_1|+|X_{2^n}|\leq2\cdot\sigma(X).$$
Let $Y:=\A_{2^n}X$. Since $\e^X$ is holomorphic, $Y\geq0$. 
Now, if equality holds in $(\ref{10.11.2015.2})$, then ${\varPhi}_{{2^n,1,}\hspace{.5pt}\widehat{a}}(\e^X)=\e^{X^{\T}\hspace*{.8pt}{\widehat{a}}}=1$. In other words, we have
\begin{equation}
Y^{\T}\A_{2^n}^{-1}\hspace*{.8pt}{\widehat{a}}=0,
\label{11.11.2015}\end{equation}
since $\A_{2^n}$ is symmetric. From (\ref{r1}) and (\ref{10.11.2015.A}), we get that 
\begin{align}
\A_{2^n}^{-1}{\widehat{a}}=\frac{1}{3\cdot2^{n-1}}
(2(1-&\operatorname{sgn}(X_1)),\hspace*{1pt}1+\operatorname{sgn}(X_1),\hspace*{1pt}1,1,\hspace*{1pt}\hdots\label{11.11.2015.1}\\
&\hdots,
\hspace*{1pt}1,\hspace*{1pt}1+\operatorname{sgn}(X_{2^n}),\hspace*{1pt}2(1-\operatorname{sgn}(X_{2^n})))^{\T}.\nonumber
\end{align}
From (\ref{11.11.2015}) and (\ref{11.11.2015.1}), it follows that 
$Y_{2^j}=0$ for all $j$ except at most two nonconsecutive values (say, $j_1$ and $j_2$) of $j\in\{0,1,n-1,n\}$.
That implies
\begin{equation}
X=Y_{2^{j_{{1}}}}\cdot\A^{-1}_{2^n}(\un,2^{j_{{1}}})+Y_{2^{j_{{2}}}}\cdot\A^{-1}_{2^n}(\un,2^{j_{{2}}}),
\label{11.11.2015.2}
\end{equation}
where $\A_{2^n}(\un, 2^{j})$ denotes the column of $\A_{2^n}$ indexed by $2^j$. 
From (\ref{11.11.2015.2}) and from (\ref{r1}), it follows that the values of $X$ at the elements of $\cl{D}_{2^n}$ are integral
if and only if both $X_1$ and $X_{2^n}$ are even.
\end{proof}

In particular, 
it follows trivially 
from the above 
lemma that 
\begin{co}
Every holomorphic eta quotient of weight $1/2$ and level $2^n$ is a rescaling of some holomorphic eta quotient on $\Gm_0(4)$.
\label{17.11.2015.1}\end{co}

\begin{lem}
If there exists a 
primitive holomorphic eta quotient of weight~$1/2$ and level $N=2^m3^n$, then $m\leq3$ and $n\leq2$.
\label{tlem}\end{lem}
\begin{proof}
Suppose 
$X\in\Z^{\cl{D}_N}$ is such that $\e^X$ is a 
primitive holomorphic eta quotient of weight~$1/2$ and level $N$. 

Suppose $n\geq3$.
It follows from Corollary~\ref{07.11.2015.2} that
${\varPhi}_{{N,3^n,{\mathds{1}_{{2^m}}}}}(\e^X)=\e_{3^{{j_{{0}}}}}$
for some $j_0\in\{0,1,\hdots,n\}$. Define $(i_1,i_2)\in\Z^2$ by
\begin{equation}
 (i_1,i_2):=\begin{cases}
       (0,1) &\text{if $n\leq 2j_0$}\\
       (n,n-1) &\text{otherwise}
      \end{cases}
\end{equation}
and define 
${\widehat{a}},{\widehat{a}'}
\in\Z^{\cl{D}_{3^n}}$
by
\begin{equation}
 {\widehat{a}}(3^i)
  =\begin{cases}
                        5&\text{if $i=i_1$}\\
                        2&\text{if $i=i_2$}\\
                        1&\text{otherwise}\\
                       \end{cases}
\ \text{ and } \ \ 
 {\widehat{a}'}(3^i)
 =\begin{cases}
                        2&\text{if $i=i_2$}\\
                        1&\text{otherwise.}\\
                       \end{cases}
\label{13.11.2015}\end{equation}
Then from Lemma \ref{04.11.2015.5}, it follows that both 
$f:={\varPhi}_{{N,2^m,}\hspace{.5pt}\widehat{a}}(\e^X)$ and
$g:={\varPhi}_{{N,2^m,}\hspace{.5pt}\widehat{a}'}(\e^X)$ are holomorphic.
Also, it follows from our choice of $i_1$ and $i_2$ that both $f$ and $g$ are of weight $1/2$.
Since $\e^X$ is a primitive eta quotient of level $N$,
there exists some $j\in\{0,1,\hdots,m\}$ such that $X^{{[3^n]}}_{2^j, 3^{i_{{1}}}}\neq0$.
So from (\ref{07.11.2015}) and (\ref{13.11.2015}), it follows that
${\varPhi}_{{N,2^m,}\hspace{.5pt}\widehat{a}}(\e^X)\neq
{\varPhi}_{{N,2^m,}\hspace{.5pt}\widehat{a}'}(\e^X)$,
i.~e. $f\neq g$.
Corollary~\ref{17.11.2015.1} implies that $f=f'_{d_1}$ and $g=g'_{d_2}$, where $f'$ and $g'$ are two primitive holomorphic quotients of weight~$1/2$ on $\Gm_0(4)$
and $d_1,d_2\in\N$. 
Here, $f'_{d_1}$ (resp.~$g'_{d_2}$) denotes the rescaling of $f'$ (resp.~the rescaling of $g'$) by $d_1$ (resp.~$d_2$).
It is easy to check the following fact:
\begin{align}
&\text{\emph{The set of primitive holomorphic eta quotients of weight~$1/2$}}\label{easy}\\[-.29em]
&\text{\emph{ on $\Gm_0(4)$ is a subset of \hyperlink{Zlist}{{\text{Zagier's list}}}.}}\nonumber 
\end{align}
From (\ref{13.11.2015}), we see that the corresponding exponents of $f$ and $g$ are congruent modulo~4.
Since
none of the exponents in the eta quotients in \hyperlink{Zlist}{{\text{Zagier's list}}} 
is a multiple of~4, it follows that $d_1=d_2$. So, the corresponding exponents of $f'$ and $g'$ are congruent modulo~4.
But again from \hyperlink{Zlist}{{\text{Zagier's list}}}, we see that
there is no pair of primitive holomorphic eta quotients of weight~$1/2$ on $\Gm_0(4)$
whose corresponding exponents are congruent modulo~4. 
Thus, we get a contradiction! Hence, we conclude that $n\leq2$.

Suppose $m\geq4$. 
For $j\in\{0,1,\hdots,n\}$,
define $
\delta_j\in\Z^{\cl{D}_{3^n}}$
by
\begin{equation}
 \delta_j(3^i)=\begin{cases}
                        1&\text{if $i=j$}\\
                        0&\text{otherwise.}\\
                       \end{cases}
\label{14.11.2015}\end{equation}
It follows from Corollary~\ref{07.11.2015.1} that for all $j\in\{0,1,\hdots,n\}$ and for all $r\in\{0,1\}$, the eta quotients
\begin{equation}
 f_{{(j,r)}}:={\varPhi}_{{N,2^m,\mathds{1}_{3^n}+r\delta_j}}(\e^X)
\end{equation}
are holomorphic.
Let $Y\in\Z^{\cl{D}_{3^n}}$ be such that
$\e^Y={\varPhi}_{{N,3^n,{\mathds{1}_{2^m}}}}(\e^X)$. 
It is easy to see that the weight of $f_{{(j,r)}}$ is $(1+Y_j)/2$.
From Corollary~\ref{07.11.2015.2}, we know that there exists a $j_0\in\{0,1,\hdots,n\}$
such that
$\e^Y=~\e_{3^{{j_{{0}}}}}$.
So, $f_{{(j,r)}}$ is a holomorphic eta quotient of weight $1/2$ on $\Gm_0(2^m)$
if and only if $(j,r)\neq(j_0,1)$. 
Let $g:=f_{{(0,0)}}$. 
Corollary~\ref{17.11.2015.1}
implies in particular, that
each holomorphic eta quotient of weight $1/2$ on $\Gm_0(2^m)$ is either nonprimitive
or not of level $2^m$.

First consider the case where $g$ is neither primitive nor of level $2^m$.
Then for $j\neq j_0$, each of the eta quotients $g_{{(j)}}:=f_{{(j,1)}}/g$
is either nonprimitive or not of level $2^m$. 
Now, if for all $j\neq j_0$, the eta quotient $g_{{(j)}}$ is nonprimitive (resp. not of level $2^m$), then 
$\e^X$ is nonprimitive (resp. not of level $2^m$) which is contrary to our assumption!
In particular, that is the case 
if $n<2$. So, $n=2$. Moreover,
there 
are distinct $j_1, j_2\in\{0,1,2\}\smallsetminus\{j_0\}$ 
such that the eta quotient $g_{{(j_1)}}$ is nonprimitive and of level~$2^m$,
whereas the eta quotient $g_{{(j_2)}}$ is primitive but not of level~$2^m$.
Since both $f_{{(j_1,1)}}$ and $f_{{(j_2,1)}}$ are of weight~$1/2$, 
it follows from Lemma~\ref{l2} that 
\begin{equation}
f_{{(j_1,1)}}=h_2\hspace{1pt}\e_{2^{{m}}}^{\pm\epsilon_1}, \ \  f_{{(j_2,1)}}=\e^{\pm\epsilon_2}h'_2  \ \text{ and } \ g_{{(j_0)}}= \e^{\mp\epsilon_2}h''_2\hspace*{1pt}\e_{2^{{m}}}^{\mp\epsilon_1},
\label{17.11.2015}\end{equation}
where  $h_2$ (resp.~$h'_2,h''_2$) denotes the rescaling of an eta quotient $h$ (resp.~$h',h''$) on $\Gm_0(2^{m-2})$ by $2$ and $\epsilon_1,\epsilon_2\in\{1,2\}$.
The last equality in (\ref{17.11.2015}) follows,
because in the case 
which we are presently considering, 
the eta quotient
$g=g_{{(j_0)}}f_{{(j_1,1)}}f_{{(j_2,1)}}$ is neither primitive nor of level $2^m$.
Let $\ell\in\{1,2\}$ be such that $|j_{{0}}-j_{{\ell}}|=1$.
It follows from Lemma~\ref{04.11.2015.5} that the eta quotient
\begin{equation}
\alpha:={\varPhi}_{{N,2^m,\mathds{1}_{3^n}}+\delta_{{j_0}}+4\delta_{{j_{{\ell}}}}}(\e^X)
\end{equation}
is holomorphic. It is easy to see that 
the weight of $\alpha$ is $(1+Y_{{j_0}}+4Y_{{i_0}})/2=1$.
Since $\alpha=f^4_{{(j_{{\ell}},1)}}\cdot g_{{(j_{{0}})}}$
is a holomorphic eta quotient of weight $1$ and level $2^m$, 
Lemma~\ref{l2} implies that 
\begin{equation}
3|\epsilon_\ell|+|\epsilon_{3-\ell}|\leq4.
\end{equation}
The only solution to the above inequality in $\epsilon_\ell,\epsilon_{3-\ell}\in\{1,2\}$ is $\epsilon_\ell=\epsilon_{3-\ell}=1$,  which contradicts
Lemma~\ref{l2}\hspace{1.8pt}!
Hence, either the eta quotient $g$ is nonprimitive and of level~$2^m$ or
$g$ is primitive but not of level~$2^m$.

Let $Z\in\Z^{\cl{D}_{2^m}}$ be such that
$\e^Z=g$. Then there exists a unique $r\in\{0,m\}$
such that $Z_{2^r}=0$.
Let $s,t\in\{0,1,\hdots,m\}$ be such that $|r-s|=1$ and
$t=m-r$. 
Then $Z_{2^t}$ is nonzero.
Since $m\geq4$, we have $|s-t|\geq3$.
Since $Z_{2^t}$ is nonzero,  it follows from 
Corollary~\ref{17.11.2015.1} and (\ref{easy}) that $Z_{2^{s}}=0$.
For $j\in\{0,1,\hdots,m\}$,
define $
\delta'_j\in\Z^{\cl{D}_{2^m}}$
by
\begin{equation}
 \delta'_j(2^i)=\begin{cases}
                        1&\text{if $i=j$}\\
                        0&\text{otherwise.}\\
                       \end{cases}
\label{19.11.2015.2}\end{equation}
Define 
${\widehat{b}},{\widehat{b}'}
\in\Z^{\cl{D}_{2^m}}$
by 
\begin{align}
{\widehat{b}}&:=(1+ |Z_{2^t}|)\mathds{1}_{2^m}+\delta'_{r}+\delta'_{s}-\operatorname{sgn}(Z_{2^t})\delta'_{t},\label{19.11.2015}\\
{\widehat{b}'}&:=(1+ |Z_{2^t}|)\mathds{1}_{2^m}+3\delta'_{r}+\delta'_{s}-\operatorname{sgn}(Z_{2^t})\delta'_{t}.
\label{19.11.2015.1}
\end{align}
Then from Lemma~\ref{04.11.2015.5}, it follows that both of the eta quotients
$\beta:={\varPhi}_{{N,3^n,}\hspace{.5pt}\widehat{b}}(\e^X)$ and
$\gamma:={\varPhi}_{{N,3^n,}\hspace{.5pt}\widehat{b}'}(\e^X)$ are holomorphic.
Also, it follows from our choice of $r,s$ and~$t$ that both $\beta$ and $\gamma$ are of weight $1/2$.
Since $\e^X$ is a primitive eta quotient of level $N$,
there exists some $j\in\{0,1,\hdots,n\}$ such that $X^{{[2^m]}}_{2^r, 3^{j}}\neq0$.
So from (\ref{07.11.2015}), (\ref{19.11.2015}) and (\ref{19.11.2015.1}),
it follows that
${\varPhi}_{{N,3^n,}\hspace{.5pt}\widehat{b}}(\e^X)\neq
{\varPhi}_{{N,3^n,}\hspace{.5pt}\widehat{b}'}(\e^X)$,
i.~e. $\beta\neq\gamma$.
From (\ref{19.11.2015}) and (\ref{19.11.2015.1}), we see that the corresponding exponents of $\beta$ and $\gamma$ are congruent modulo~2. Thus, we have obtained two distinct holomorphic eta quotients of weight $1/2$ 
on $\Gm_0(3^n)$ whose corresponding exponents are congruent modulo~2. This contradicts
Corollary~\ref{19.11.2015.3}\hspace{1.8pt}! Hence, we conclude that $m\leq3$.
\end{proof}

\section*{Acknowledgments}
I would like to thank Don Zagier for his encouragement in writing up this article
and for his comments 
on an earlier version of the manuscript. 
I~am grateful to the Max Planck Institute for Mathematics in Bonn and to
the CIRM~:~FBK (International Center for Mathematical Research of the Bruno Kessler Foundation) in Trento
for providing me with office spaces and
supporting me with 
fellowships 
during the preparation of this article.

\bibliography{wt-half-bibtex}
\bibliographystyle{IEEEtranS}
\nocite{*}

 \end{document}